\documentclass{sig-alternate-arxiv}

\usepackage{amssymb}
\usepackage[footnotesize,it]{caption}

\newif\ifwithexperiments
\withexperimentsfalse

\numberwithin{equation}{section}
\numberwithin{figure}{section}
\usepackage{microtype}
\usepackage{amssymb,amsmath}
\usepackage{graphicx}
\usepackage{rotating}
\usepackage{mathptmx}

\usepackage{epsfig}
\usepackage{mathbbol}
\usepackage{algorithmic}
\usepackage{url}
\usepackage{wrapfig}
\usepackage{algorithm}

\newcommand{\floor}[1]{\lfloor #1 \rfloor}
\newcommand{\ceil}[1]{\lceil #1 \rceil}

\newcommand{\donotshow}[1]{}

\newcommand{\ignore}[1]{}

\newtheorem{theorem}{Theorem}
\newtheorem{lemma}[theorem]{Lemma}

\newtheorem{corollary}[theorem]{Corollary}
\newtheorem{definition}{Definition}
\providecommand{\qed}{\rule[-0.2ex]{0.3em}{1.4ex}}

\providecommand{\R}{\mathbb{R}}
\newcommand{\N}{\mathbb{N}}
\newcommand{\Z}{\mathbb{Z}}
\newcommand{\Q}{\mathbb{Q}}
\newcommand{\C}{\mathbb{C}}

\newlength{\setspacing}
\providecommand{\set}[2]{ \left\{\, #1 \mbox{{\rm ; }} #2 \, \right\}  }

\newcommand{\mbegin}{\{\ \ }
\newcommand{\mend}{\}}

\newlength{\mleftindent}
\newlength{\mindent}
\newlength{\mboxwidth}
\newcommand{\mincrement}{\addtolength{\mboxwidth}{-\mindent}}
\newcommand{\mdecrement}{\addtolength{\mboxwidth}{\mindent}}

\newlength{\preprogramskip}
\newlength{\postprogramskip}
\newlength{\mexpwidth}
\newlength{\mexpindent}
\newcommand{\indentafterkeyword}{\hspace*{0.5em}}

\newcommand{\mslifelse}[3]  
{\setlength{\mexpwidth}{\mboxwidth}%
\settowidth{\mexpindent}{{\bf if\indentafterkeyword}}%
\addtolength{\mexpwidth}{-\mexpindent}%
{\bf if\indentafterkeyword}\parbox[t]{\mexpwidth}{#1}\\
\mincrement \mbegin \parbox[t]{\mboxwidth}{#2 \mend} \mdecrement \\
{\bf else} \\
\mincrement \mbegin \parbox[t]{\mboxwidth}{#3}\\
\mend \mdecrement
}

{\vspace{-0.3em}\begin{itemize}
\setlength{\itemsep}{0.2\itemsep}%
\setlength{\parskip}{0.2\parskip}%
\setlength{\topsep}{0.0\topsep}%
}
{\end{itemize}}

{\begin{itemize}
\setlength{\itemsep}{0.2\itemsep}%
\setlength{\parskip}{0.2\parskip}%
\setlength{\topsep}{0.0\topsep}%
}
{\end{itemize}}

{\vspace{-0.3em}\begin{description}
\setlength{\itemsep}{0.2\itemsep}%
\setlength{\parskip}{0.2\parskip}%
\setlength{\topsep}{0.0\topsep}%
}
{\end{description}}

{\begin{description}
\setlength{\itemsep}{0.2\itemsep}%
\setlength{\parskip}{0.2\parskip}%
\setlength{\topsep}{0.0\topsep}%
}
{\end{description}}

{\vspace{-0.3em}\begin{enumerate}
\setlength{\itemsep}{0.2\itemsep}%
\setlength{\parskip}{0.2\parskip}%
\setlength{\topsep}{0.0\topsep}%
}
{\end{enumerate}}

{\begin{enumerate}
\setlength{\itemsep}{0.2\itemsep}%
\setlength{\parskip}{0.2\parskip}%
\setlength{\topsep}{0.0\topsep}%
}
{\end{enumerate}}

\newlength{\proofpostskipamount}\newlength{\proofpreskipamount}
               {\par\vspace{0.5ex}\noindent{\bf Proof #1:}\hspace{0.5em}}%
               {\nopagebreak%
                \strut\nopagebreak%
                \hspace{\fill}\qed\par\medskip\noindent}

\newlength{\mydefwidth}
\newlength{\mytextwidth}

\newcommand{\myurl}[1]{{\footnotesize \url{#1}}}

\newcommand{\multipoint}[2]{#1[#2]}
\newcommand{\var}{\operatorname{var}}

\urldef{\mails}\path|{msagralo}@mpi-inf.mpg.de|

\begin{document}

\title{A Near-Optimal Algorithm for Computing Real Roots of Sparse Polynomials}

\numberofauthors{1} 

\author{
%
\alignauthor
Michael Sagraloff\\
       \affaddr{Max-Planck-Institut f\"ur Informatik, Germany}\\
}

\maketitle

\begin{abstract}
Let $p\in\Z[x]$ be an arbitrary polynomial of degree $n$ with $k$ non-zero integer coefficients of absolute value less than $2^\tau$. In this paper, we answer the open question whether the real roots of $p$ can be computed with a number of arithmetic operations over the rational numbers that is polynomial in the input size of the sparse representation of $p$.
More precisely, we give a deterministic, complete, and certified algorithm that determines isolating intervals for all real roots of $p$ with $O(k^3\cdot\log(n\tau)\cdot \log n)$ many exact arithmetic operations over the rational numbers.

When using approximate but certified arithmetic, the bit complexity of our algorithm is bounded by $\tilde{O}(k^4\cdot n\tau)$, where $\tilde{O}(\cdot)$ means that we ignore logarithmic. Hence, for sufficiently sparse polynomials (i.e.~$k=O(\log^c (n\tau))$ for a positive constant $c$), the bit complexity is $\tilde{O}(n\tau)$. We also prove that the latter bound is optimal up to logarithmic factors. 
\end{abstract}

\section{Introduction}\label{intro}

Throughout the following considerations, let
\begin{align}
p(x):=\sum\nolimits_{i=0}^n a_i x^i\in\Z[x]\text{, with }|a_i|<2^\tau\text{ and }\tau\in N_{\ge 1},\label{polyf}
\end{align}
be a (not necessarily square-free) polynomial of degree $n$ with integer coefficients of bit-size less than $\tau$, and let $k$ be the number of non-zero coefficients $a_{i_0},\ldots,a_{i_{k-1}}$, with $0\le i_{0}\le\cdots\le i_{k-1}=n$. For convenience, we denote a polynomial $p\in\Z[x]$ of degree at most $n$ and with at most $k$ non-vanishing coefficients, each of absolute value less than $2^\tau$, a \emph{$k$-nomial of magnitude} $(n,\tau)$.
We assume that $p$ is given by its sparse representation
\begin{align}
p(x)=\sum\nolimits_{l=0}^{k-1} a_{i_l} x^{i_l},\quad\text{where }a_{i_l}\neq 0\text{ for all }l=0,\ldots,k-1.\label{polyfsparse}
\end{align}
Notice that the sparse representation needs $O(k\cdot(\tau+\log n+1))$ many bits. Namely, we need one bit for the sign of each coefficient $a_{i_l}$, $\tau$ or less bits for the binary representation of $|a_{i_l}|$, and $\log n$ bits for the binary representation of each index $i_l$. To date, it was unknown whether we can isolate (or just count) the real roots of $p$ with a number of arithmetic operations over $\Q$ that is polynomial in the input size of the sparse representation of $p$. This paper gives a positive answer to the latter question. In addition, we show that, for isolating all real roots of a sparse enough polynomial $p\in\Z[x]$, our algorithm is near-optimal:

\begin{theorem}
Let $p\in\Z[x]$ be a $k$-nomial of magnitude $(n,\tau)$, then we can isolate all real roots of $p$ with $O(k^3\cdot \log(n\tau)\cdot\log n)$ many arithmetic operations over the rational numbers. In addition, for $k=O(\log^c(n\tau))$, with $c$ a non-negative constant, we need at most $\tilde{O}(n\tau)$ bit operations to isolate all real roots of $p$. The latter bound is optimal up to logarithmic factors in $n$ and $\tau$.
\end{theorem}

There exist numerous algorithms,\footnote{\small The literature on root solving is extensive. Hence, due to space limitations, we decided to restrict to a small selection of representative papers and refer the reader to the references given therein.} e.g.~\cite{DBLP:journals/jsc/BurrK12,ESY06,DBLP:conf/issac/GarciaG12,McNamee-Pan,MSW-rootfinding2013,Pan:alg,rouillier-zimmermann:roots:04,Sagraloff12,DBLP:journals/corr/SagraloffM13,DBLP:conf/issac/YapS11,Schoenhage,DBLP:journals/tcs/Tsigaridas13}, for efficiently computing the real (complex) roots of a polynomial $p$ as in (\ref{polyfsparse}), \emph{given that $k$ is large enough}. 
That is, for $k=\Omega(n^c)$ with $c$ an arbitrary but positive constant, the computational complexity of 
these algorithms is polynomial in the input size. For isolating all complex roots of $p$, Pan's method~\cite{MSW-rootfinding2013,Pan:alg}, which goes back to Sch\"onhage's splitting circle approach~\cite{Schoenhage}, achieves 
record bounds with respect to arithmetic and bit complexity in the worst case. More specifically, it needs $\tilde{O}(n)$ 
arithmetic operations performed with a precision of $\tilde{O}(n\tau)$ bits, and thus, $\tilde{O}(n^2\tau)$ bit operations. Besides Pan's method, which computes all complex roots at once, there 
also exist very efficient methods for computing the real roots only. A recently proposed algorithm, denoted \textsc{ANewDsc}~\cite{DBLP:journals/corr/SagraloffM13}, which combines Descartes' Rule of Signs, Newton iteration, and approximate arithmetic, has a bit complexity that is comparable to Pan's method; for any given positive integer $L$, \textsc{ANewDsc} uses $\tilde{O}(n^3+n^2\tau+nL)$ bit operations to compute isolating intervals of size less than $2^{-L}$ for all real roots of $p$.
We further remark that both of the above mentioned methods can be used to efficiently isolate the roots of a polynomial $p$ whose coefficients can only be learned from (arbitrarily good) approximations, given that $p$ has no multiple roots.
In this model, the bound on the bit complexity is stated in terms of the degree, the discriminant, and the Mahler bound of~$p$.

In contrast, for general $k$, much less is known about the computational complexity of computing (or just counting) the real roots of $p$. In~\cite{CUCKER1999}, Cucker et al. proposed a method to compute all \emph{integer} roots of $p$ with a number of bit operations that is polynomial in the input size. Lenstra~\cite{lenstra99} further showed that all rational roots of $p$ can be computed in polynomial time. In fact, he even proved that one can compute all factors of $p$ over $\Q$ of a fixed degree $d$ with a number of bit operations that is polynomial in the input size and $d$. For trinomials $p$ (i.e.~$k=3$) with arbitrary real coefficients, Rojas and Ye~\cite{Rojas05} gave an algorithm for counting (and \emph{$\epsilon$-approximating}) all real roots of $p$ that uses $O(\log^2 n)$ arithmetic operations in the field over $\Q$ generated by the coefficients of $p$. However, already for polynomials $p\in\R[x]$ with more than $3$ monomials, it is unknown whether there exists a deterministic polynomial-time algorithm for computing (or just counting) the real roots of $p$. Bastani et al.~\cite{Bastani11} introduced a deterministic algorithm that, for most inputs, counts the number of real roots of a tetranomial $p$ (i.e.~$k=4$). Its arithmetic complexity is polynomial in the input size, and, in the special case where $p$ has integer coefficients, even the bit complexity is polynomial.
For general $k$-nomials $p\in\Z[x]$ with integer coefficients, we are not aware of any method, either for counting or isolating the real roots, that achieves an arithmetic complexity that is polynomial in the input size of the sparse representation of $p$. 

For the bit complexity, the best known bound for isolating the roots of a (not necessarily sparse polynomial) $p\in\Z[x]$ is $\tilde{O}(n^2\tau)$, and we expect no improvement of the corresponding algorithms when restricting to sparse polynomials. Namely, since Pan's method computes all complex roots, it needs $\Omega(n)$ arithmetic operations. Also, methods based on Descartes' Rule of Signs need at least $\Omega(n)$ arithmetic operations due to a transformation of the polynomial $p$ that destroys its sparsity. This together with the fact that there exist $4$-nomials that require a precision of $\Omega(n\tau)$ for isolating its (real) roots (see also the proof of Theorem~\ref{maintheorem2}) indicates that both approaches have a worst-case bit complexity of $\Omega(n^2\tau)$. In addition, for isolating the roots of $p$, most algorithms need to compute the square-free part of $p$ in a first step, and the best known deterministic bounds~\cite[Section~14]{gathen-gerhard:algebra:bk} for the arithmetic and bit complexity of the latter problem are $\tilde{O}(n)$ and $\tilde{O}(n^2\tau)$, respectively.\\

Our algorithm is rather simple from a high-level perspective and combines mainly known techniques. Thus, we consider our contribution to be the right assembly of these techniques into an algorithm and the complexity analysis.
The main idea underlying our approach is to compute isolating intervals for the roots of $p_0:=p/x^{i_0}$ from sufficiently small isolating intervals for the roots of the polynomial\footnote{\small For simplicity, the reader may assume that $i_0=0$, and thus, $p_1$ has the same roots as the derivative $p'=\frac{dp}{dx}$ of $p$ except for the root at zero.} $p_1:=p_0'\cdot x^{1-i_1}$. Notice that $p_1$ is a $(k-1)$-nomial of magnitude $(n,\tau+\log n)$ with a non-vanishing constant coefficient. Using evaluation and separation bounds, we can determine the sign of $p_0$ at the roots of $p_1$ by evaluating $p_0$ at arbitrary points in the corresponding isolating intervals, and thus, we can compute common roots of $p_0$ and $p_1$. In addition, we can immediately derive isolating intervals for the simple real roots of $p_0$ as $p_0$ is monotone in between two consecutive roots of $p_1$. Then, the isolating intervals for the roots of $p_0$ can be further refined to an arbitrary small size. Hence, recursive application of the above approach allows us to compute isolating intervals for $p$ from the roots of a $1$-nomial $p_{k-1}\in\Z[x]$ after $k$ iterations; see Section~\ref{sec:algorithm} for details.

Efficiency of the above approach crucially depends on the method to refine the isolating intervals for the simple roots of the polynomials $p_i$ that are considered in the recursion.
For this, we modify an efficient method for approximating (clusters of) real roots as recently proposed in~\cite{DBLP:journals/corr/SagraloffM13}. Since the method from~\cite{DBLP:journals/corr/SagraloffM13} is based on Descartes' Rule of Signs, its arithmetic complexity is super-linear in $n$. Hence, in order to exploit the sparsity of the polynomials $p_i$, we had to replace the corresponding steps by simple polynomial evaluation.
 For an arbitrary positive integer $L$, the so-obtained method refines arbitrarily isolating intervals for all simple roots of a $k$-nomial $p$ of magnitude $(n,\tau)$ to a size less than $2^{-L}$ in $O(k\cdot (\log n+ \log(\tau+L)))$ iterations, and, in each iteration, $p$ is evaluated at a constant number of points. This yields an arithmetic complexity for the refinement steps that is polynomial in the input size. We consider the refinement method as the key ingredient of our algorithm, and think that it is of independent interest.
 
When using exact arithmetic over the rationals, the bit complexity of our algorithm is $\tilde{O}(k^3\cdot n^2\tau)$. We further show that, when replacing exact by approximate computation, the bit-size of the intermediate results reduces by a factor $n$ for the price of using $k$ times as many arithmetic operations. This yields the bound $\tilde{O}(k^4\cdot n\tau)$, and thus, $\tilde{O}(n\tau)$ for sufficiently small $k$, that is, $k=O(\log^c(n\tau))$. We also prove that the latter bound is optimal, where we use the fact that there exist $4$-nomials such that the binary representations of the corresponding isolating intervals need $\Omega(n\tau)$ bits. 

\section{Algorithm and Complexity}\label{sec:idea}

\subsection{The Algorithm}\label{sec:algorithm}

\ignore{Let $z_1,\ldots,z_n$ denote the complex roots of $p$, $\sigma(z_i):=\min_{j\neq i}|z_i-z_j|$ the \emph{separation of $z_i$}, and $\sigma_p:=\min_i \sigma(z_i)$ the \emph{separation of $p$}.}
It is well known (e.g., see~\cite{CUCKER1999,Rojas05}) that the number of real roots of $p$ is upper bounded by $2k-1$. Namely,
according to Descartes' Rule of Signs, the number of positive real roots of $p$ (counted with multiplicity) is upper bounded by the number of sign changes in the coefficient sequence of $p$, and thus, smaller than $k$. The same argument applied to the polynomial $p(-x)$ further shows that the number of negative roots of $p(x)$ is smaller than $k$ as well. 

In what follows, we may assume, w.l.o.g., that $i_0=0$. Namely, if $i_0>0$, then $p$ has the same roots as $p/x^{i_0}$ plus an additional root at $x=0$ of multiplicity $i_0$. Hence, we can consider the polynomial $p/x^{i_0}$ instead. In addition, we may restrict our search to the positive roots; for the negative roots, we can then apply the same approach to the polynomial $p(-x)$. 
According to Cauchy's root bound, the modulus of each (complex) root is upper bounded by $1+2^{\tau}< 2^{\tau+1}$, and thus, for isolating the positive real roots of $f$, we can restrict our search to the  interval $\mathcal{I}:=(0,2^{\tau+1})$. 
We write $p$ as
\begin{align*}
p(x)=a_{i_0}+x^{i_{i_1}}\cdot(a_{i_1}+\cdots+a_{i_{k-1}}\cdot x^{i_{k-1}-i_1})=a_{i_0}+x^{i_1}\cdot\hat{p}(x),
\end{align*}
where $\hat{p}$ has degree $n-l_1<n$ and exactly $k-1$ non-zero coefficients. The idea is now to compute isolating intervals for the positive roots of $p_0:=p$ from sufficiently small isolating intervals for the positive roots of its derivative $p'(x):=\frac{dp(x)}{dx}$. For this, we do not directly consider the derivative $p'(x)$ but the polynomial
\begin{align}
p_1(x):=x\cdot\hat{p}'(x)+i_1\cdot\hat{p}(x)=\frac{p'(x)}{x^{i_1-1}},\label{polybarf}
\end{align}
which has the same roots as $p'$ except for the root at $x=0$ of multiplicity $i_1-1$. Notice that $p_1$ is a $(k-1)$-nomial of magnitude $(n-i_1,\tau+\log n)$ and that its coefficients can be computed from the coefficients of $p$ using $k$ multiplications and $k$ additions. 

Let $x_{1}'$ to $x_{k_1}'$, with $0\le k_1\le k-2$, denote the roots of $p_1$ that are contained in $\mathcal{I}=(0,2^{\tau+1})$. W.l.o.g., we may assume that $0<x_{1}'<x_{2}'<\cdots<x_{k_1}'<2^{\tau+1}$.
Now, suppose that, for each root $x_{j}'$, an isolating interval $I_{j}'=(a_{j},b_{j})\subset\mathcal{I}$ of width less than $2^{-L}$, with $L:=128\cdot n\cdot(\tau+k\cdot \log n)$ and $a_{j},b_{j}\in\Q$, is given. Then, based on the following theorem, we can compute the sign of $p$ at the roots of $p_1$, and thus, determine common roots of $p$ and $p_1$. 
Because of space limitations, we give the proof of Theorem~\ref{evalbound} in the Appendix. It mainly combines known results, however, we remark that we could not find a comparable result in the literature, where only evaluation/separation bounds of size $\tilde{O}(n^2+n\mu)$ are given. 

\begin{theorem}\label{evalbound}
Let $f$ and $g$ be polynomials of degree $n$ or less with integer coefficients of absolute values less than $2^{\mu}$, and let 
\begin{align}
L:=128\cdot n\cdot (\log n+\mu).
\end{align}
Then, for any two distinct roots $\xi_i$ and $\xi_j$ of $F:=f\cdot g$, it holds that $|\xi_i-\xi_j|^{m_i}>2^{-L}$, where $m_i:=\operatorname{mult}(\xi_i,F)$ denotes the multiplicity of $\xi_i$ as a root of $F$. If $\xi$ is a root of $g$ and $f(\xi)\neq 0$, then it holds that $|f(x)|>2^{-L/4}$ for all $x\in\C$ with $|x-\xi|<2^{-L}$. Vice versa, if $f(\xi)=0$, then $|f(x)|<2^{-L}$ for all $x\in\C$ with $|x-\xi|<2^{-L}$.
\end{theorem}

From the above theorem, we conclude that, for all $j=1,\ldots,k_1$:\smallskip

\noindent $\bullet$  $p$ has at most one root $\xi$ in $I_{j}'$.\\
\noindent $\bullet$  If $p$ has a root $\xi$ in $I_{j}'$, then $\xi=x_{j}'$ and $|p(x)|<2^{-L}$ for all $x\in I_{j}'$.\\ 
\noindent $\bullet$  If $p$ has no root in $I_{j}'$, then $|p(x)|>2^{-L/4}$ for all $x\in I_{j}'$.\smallskip

Hence, we can determine the sign of $p$ at each root $x_{j}'$ of $p_1$ by evaluating $p(x)$ to an absolute error\footnote{\small For now, you may assume that we exactly evaluate $p(x)$ for some rational $x\in I_{j}'$. However, we will need the more general statement for our version of the algorithm that uses approximate arithmetic, as proposed in Section~\ref{bitcomplexity}.} of less than $2^{-L/2}$, where $x$ is an arbitrary point in $I_{j}'$. 
Let $x_{0}':=0$ and $x_{k_1+1}':=2^{\tau+1}$, and let $I_{0}'=[a_{0},b_{0}]:=[x_{0}',x_{0}']$ and $I_{k_1+1}'=[a_{k_1+1},b_{k_1+1}]:=[x_{k_1+1}',x_{k_1+1}']$ be corresponding intervals of width $0$. Notice that the values $x_{j}'$ decompose $\mathcal{I}$ into $k_1+1$ many intervals $A_j:=(x_{j-1}',x_{j}')$ such that $p$ is monotone in each interval~$A_j$. In addition, if either $p(x_{j-1}')=0$ or $p(x_{j}')=0$, then $p$ has no root in $A_j$ according to Rolle's Theorem. 
Hence, $p$ has a root $\xi$ in $A_j$ if and only $p(x_{j-1}')\cdot p(x_{j}')<0$.
 If the latter inequality holds, then $\xi$ is unique and simple. In fact, it even holds that the shortened  interval $A_{j}':=(a_{j-1},b_{1,j})\subset A_j$ isolates $\xi$ because $I_{j-1}'$ and $I_{j}'$ do not contain any root of $p$. Now, since we can compute the sign 
of $p$ at all  points $x_{j}'$, isolating intervals for the positive real roots of 
$p$ can directly be derived from the intervals $I_{j}'$. Notice that, for the 
positive roots of $p$ with multiplicity larger than $1$, isolating intervals of width 
less than $2^{-L}$ are already given. Namely, the multiple roots of $p$ are exactly 
the common roots of $p$ and $p_1$, and thus they are already isolated by some of the intervals 
$I_{j}'$. Each simple positive root of $p$ is isolated by an interval $A_j'$, which can be further refined to a width less than $2^{-L}$ using the refinement method
from Section~\ref{sec:refinement}. In summary, we have shown how to compute isolating intervals of width less than $2^{-L}$ for all roots of $p$ contained in $\mathcal{I}$ from isolating intervals of width less than $2^{-L}$ for all roots of~$p_1$ that are contained in $\mathcal{I}$.

We now recursively apply the above approach to $p_1$. More explicitly, for $j=1,\ldots,k-1$, we first compute the polynomials
\begin{align*}
&p_0:=p,\text{ }p_{j}:=x\cdot\hat{p}_{j-1}'+(i_{j}-i_{j-1})\cdot \hat{p}_j=x^{-(i_{j}-i_{j-1})+1}\cdot p_{j-1}'(x),\\&\text{where } p_{j-1}=p_{j-1}(0)+x^{i_{j}-i_{j-1}}\cdot \hat{p}_{j-1}(x),\text{ and }\hat{p}_{j-1}(0)\neq 0.
\end{align*}
Since $p_j$ is a $(k-j)$-nomial of magnitude 
$(n-i_j,\tau+j\cdot \log n)$, $p_{j}$ becomes a constant for  
$j=k-1$. Thus, computing isolating intervals of width less than $2^{-L}$ for the 
positive roots of $p_{k-1}$ is trivial. Going backwards from $j=k-1$, we can then 
iteratively compute isolating intervals $I_{j-1,1}$ to $I_{j-1,k_j}$ of width less than 
$2^{-L}$ for the roots of $p_{j-1}$ from isolating intervals $I_{j,1}$ to $I_{j,k_{j}}$ of width less than $2^{-L}$ for the roots of $p_{j}$. Notice that Theorem~\ref{evalbound} applies to $f:=p_{j-1}$, $g:=p_{j}$ and any point $x\in I_{j,i}$ 
as $I_{j,i}$ has width less than $2^{-L}$ and $L=128\cdot n\cdot (n+k\cdot \log n)\ge 128\cdot\max(\deg(p_{j-1}),\deg(p_j))\cdot(\log\max(\|p_{j-1}\|_\infty,\|p_{j}\|_\infty)+\log n)$ for all $j\le k-1$. Hence, we can compute the sign of $p_{j-1}$ at each positive root of $p_{j}$ by evaluating $p_{j-1}$ at an arbitrary point in the corresponding isolating interval. 

Notice that the above approach does not only yield isolating intervals for all real roots of $p$ but also the corresponding multiplicities. Namely, a root $\xi$ of $p$ has multiplicity $j$ if and only if $p_0(\xi)=\cdots=p_{j-1}(\xi)=0\neq p_j(\xi)$.\\
   
Before we continue with the analysis of our algorithm, we first consider a simple example to illustrate our approach. Notice that we tried to keep the formulation of our algorithm as simple as possible with the prior goal to achieve the claimed complexity bounds, however, for the cost of a probably worse performance in practice. Hence, for an actual implementation, we propose to integrate additional steps in order to avoid costly refinement steps, and thus, to considerably speed up the overall approach. We hint to such techniques in the following section. The reader who is mainly interested in the theoretical complexity bounds should feel free to skip the example and directly continue with Section~\ref{ssec:complexity1}.

\subsection{An Example and Alternatives}

Let $p(x)=x^{50}-4\cdot x^{48}+4\cdot x^{46}-x^4+4\cdot x^2-4$ be a $6$-nomial of magnitude $(n,\tau):=(50,2)$. We consider the polynomials $p_j$, with $j=0,\ldots,5$, that are defined as follows:
\begin{align*}
p_0(x)&:=x^{50}-4\cdot x^{48}+4\cdot x^{46}-x^4+4\cdot x^2-4\\
&=x^2\cdot (x^{48}-4\cdot x^{46}+4\cdot x^{44}-x^2+4)-4=x^2\cdot\hat{p}_0-4\\
p_1(x)&:=x\cdot\hat{p}_0'+2\cdot \hat{p}_0=x^{-1}\cdot p_0'(x)=\\
&=x^2\cdot(50\cdot x^{46}-192\cdot x^{44}+184\cdot x^{42}-4)+8=x^2\cdot\hat{p}_1+8\\
p_2(x)&:=x\cdot\hat{p}_1'+2\cdot \hat{p}_1=x^{-1}\cdot p_1'(x)=\\
&=x^{42}\cdot(2400\cdot x^4-8832\cdot x^2+8096)-8=x^{42}\cdot\hat{p}_2-8\\
p_3(x)&:=x\cdot\hat{p}_2'+42\cdot \hat{p}_2=x^{-41}\cdot p_2'(x)=\\
&=x^2\cdot (110400\cdot x^2-388608)+340032=x^2\cdot\hat{p}_3+340032\\
p_4(x)&:=x\cdot\hat{p}_3'+2\cdot \hat{p}_3=x^{-1}\cdot p_3'(x)\\
&=x^2\cdot 441600-777216=x^2\cdot\hat{p}_4-777216\\
p_5(x)&:=x\cdot\hat{p}_4'+2\cdot \hat{p}_4=x^{-1}\cdot p_4'(x)=883200
\end{align*}

We want to recursively isolate and approximate the positive real roots of the polynomials $p_5$ to $p_0$, starting with $p_5$. Since we are only interested in the roots of $p=p_0$, we can restrict to the interval $\mathcal{I}:=(0,8)$, which must contain all positive roots of $p$. Trivially, $p_5$ has no root, and thus, $p_4$ is monotone in $\mathcal{I}$.
Since $p_4(0)<0$ and $p_4(8)>0$, the interval $I_{4,1}:=\mathcal{I}$ isolates the unique (simple) positive real root (at $x_{4,1}\approx 1.326$) of $p_4$ in $\mathcal{I}$. The polynomial $p_3$ is monotone in each of the two intervals $(0,x_{4,1})$ and $(x_{4,1},8)$. Refining the isolating interval for $x_{4,1}$ to a width less than $2^{-L}$, with $L:=128\cdot \deg(p)\cdot(\log \|p\|_{\infty}+6\cdot\log \deg(p))\approx 8.5\cdot 10^4$, and using Theorem~\ref{evalbound}, we can evaluate the sign of $p_3$ at $x=x_{4,1}$. Since $p_3(0)>0$, $p_3(x_{4,1})\approx -1943<0$, and $p_3(8)>0$, each of the two intervals $I_{3,1}:=(0,x_{4,1})$ and $I_{3,2}=(x_{4,1},8)$ isolates a (simple) positive real root (at $x_{3,1}\approx 1.275$ and at $x_{3,2}\approx 1.375$) of $p_3$. The polynomial $p_2$ is monotone in each of the three intervals $(0,x_{3,1})$, $(x_{3,1},x_{3,2})$, and $(x_{3,2},8)$.
We again refine the isolating intervals for $x_{3,1}$ and $x_{3,2}$ to a width less than $2^{-L}$ and evaluate the sign of $p_2$ at the points $x=0$, $x=8$, and at the roots of $p_3$. From the latter evaluations, we conclude that $p_2$ has exactly three positive (simple) real roots (at $x_{2,1}:=0.869\ldots$, $x_{2,2}\approx 1.315$, and at $x_{2,3}\approx 1.396$), which are isolated by the intervals $I_{2,1}:=(0,x_{3,1})$, $I_{2,2}:=(x_{3,1},x_{3,2})$, and $I_{2,3}:=(x_{3,2},8)$, respectively.
Refining the isolating intervals for $x_{2,1}$, $x_{2,2}$, and $x_{2,3}$ to a width less than $2^{-L}$ again allows us to evaluate the sign of $p_1$ at at the roots of $p_2$. The latter computation shows that $p_1$ has exactly two (simple) positive real roots in $\mathcal{I}$ (at $x_{1,1}\approx 1.356$ and at $x_{1,2}\approx 1.414$), which are isolated by the intervals $(x_{2,2},x_{2,3})$ and $(x_{2,3},8)$, respectively. Eventually, we refine the intervals to a width less than $2^{-L}$ and evaluate the sign of $p_0=p$ at the roots of $p_1$. We have $p_0(x_{1,1})\approx 3\cdot 10^4$ and $p_0(x)<2^{-L}$, where $x$ has been arbitrary chosen from the isolating interval for $x_{1,2}$. Hence, from Theorem~\ref{evalbound}, we conclude that $p_0(x_{1,2})=0$, and thus, $x_{0,1}:=x_{1,2}$ is the unique positive real root of $p$. In addition, $x_{0,1}$ has multiplicity~$2$.\medskip

Notice that, in each except the last step (i.e.~for $j=2,\ldots,5$), we could consider an alternative approach, where we simultaneously refine the isolating interval for a root $\xi$ of $p_{j}$ and use interval arithmetic to evaluate the sign of $p_{j-1}(\xi)$. Following the analysis in~\cite[Section 4]{qir-kerber-11}, one can show that, if $p_{j-1}(\xi)\neq 0$, then this approach yields the correct sign as soon as the interval has been refined to a width less than $2^{-L'}$, with $L'=\deg(p_{j-1})\cdot(4+\log\max(1,|\xi|))-\log|p_{j-1}(\xi)|+\tau$. For instance, in our example above, the sign of $p_3$ at the root $x_{4,1}$ of $p_4$ can be determined from an isolating interval for $x_{4,1}$ of width less than $2^{-11}$ (compared to the theoretical bound of approximate size $2^{- 8.5\cdot 10^4}$ from Theorem~\ref{evalbound}).
Hence, for a practical implementation, we strongly recommend to integrate such techniques to rule out easy cases in a more efficient way.
However, for deciding that $p_0$ evaluates to zero at $x=x_{1,2}$, methods that are purely based on approximate computation will not work.\footnote{\small That is, without computing an explicit theoretical evaluation bound $2^{-L}$ as given in Theorem~\ref{evalbound}. Certainly, if one is willing to use such a bound, then also numerical computation will yield the correct result as soon as the interval has size less than $2^{-L}$ and the precision of the approximate arithmetic is large enough to guarantee an absolute error of less than $2^{-L/2}$.} One possible way, as proposed in our algorithm, is to refine the isolating interval for $x_{1,2}$ to a sufficiently small size, to evaluate $p$ at an arbitrary point that is contained in the isolating interval, and to apply Theorem~\ref{evalbound}. 
Another way is to compute the square-free part $g^*$ of the greatest common divisor $g:=\gcd(p_{j-1},p_{j})$ of $p_{j-1}$ and $p_{j}$ and to check whether $g^*$ changes signs at the endpoints of the isolating interval. The advantage of the latter approach is that $p_{j-1}$ and $p_{j}$ typically do not share a common non-trivial factor, which can be easily checked via modular computation, and thus, it suffices to use interval arithmetic to compute the sign of $p_{j-1}$ at the roots of $p_{j}$. 

However, although the second approach, which is based on computing $g^*$, seems to be more efficient in practice, there is a severe drawback with respect to its arithmetic complexity. Namely, the considered symbolic computations need a number of arithmetic operations that is super-linear in the degree of the involved polynomials. In contrast, we will show that the first approach, which is entirely based on refinement and evaluation, only uses a number of arithmetic operations that is polynomial in the input size of the sparse representation of the input polynomial.

\subsection{Arithmetic Complexity}\label{ssec:complexity1}

For an arbitrary $k$-nomial $p\in\Z[x]$ of magnitude $(n,\tau)$, suppose that each simple positive root $\xi$ of $p$ is already isolated by a corresponding interval $I=(a,b)\subset\mathcal{I} =(0,2^{\tau+1})$ with rational endpoints $a$ and $b$. In Section~\ref{sec:refinement}, we give an algorithm to refine \emph{all} such isolating intervals to a width less than $2^{-L}$ using $O(k^2\cdot (\log(n\tau)+\log L)\cdot\log n)$ arithmetic operations over $\Q$.
Hence, from the definition of our algorithm for root isolation, we conclude that we can compute isolating intervals of width less than $2^{-L}$, with $L:=128\cdot n\cdot(n+k\cdot \log n)$, for all roots of $p_{j-1}$ contained in $\mathcal{I}$ from isolating intervals of size less than $2^{-L}$ for the 
roots of $p_{j}$ contained in $\mathcal{I}$ using only $O(k^2\cdot \log(n\tau)\cdot\log n)$ many 
arithmetic operations: Namely, we can compute $p_j$ from $p_{j-1}$ with $k$ multiplications and $k$ additions. For evaluating $p_{j-1}$ at an arbitrary point $x\in\Q$, we need 
at most $2k\log n$ arithmetic operations since we can compute $x^i$ with less than 
$2\log i$ multiplications by repeated squaring (e.g.~$x^{11}=x\cdot x^2\cdot 
((x^2)^2)^2$) and $p_{j-1}$ has at most $k$ non-vanishing coefficients. We have shown that 
evaluating the sign of $p_{j-1}$ at each root $\xi\in \mathcal{I}$ of $p_j$ can be reduced to the evaluation of $p_{j-1}$ 
at an arbitrary (rational) point $x\in I'$, where $I'$ is an isolating interval for $\xi$. Hence, the latter evaluations need at 
most $(k+1)\cdot (2k\log n)$ many arithmetic operations as each polynomial $p_j$ is a $(k-j)$-nomial of magnitude $(n-i_j,\tau+j\cdot\log n)$. Finally, the refinement of 
the isolating intervals for the simple positive roots of $p_{j-1}$ needs $O(k^2\cdot 
\log(n\tau)\cdot\log n)$ many arithmetic operations.  Hence, the total number of arithmetic operations is bounded by
\begin{align}\nonumber
&k\cdot\left(3k+2k^2\log n+O(k^2\cdot \log(n\tau)\cdot\log n)\right).
\end{align}

We fix this result:

\begin{theorem}\label{thm:main1}
Let $p\in\Z[x]$ be a $k$-nomial of magnitude $(n,\tau)$, then all real roots of $p$ can be isolated with $O(k^3\cdot \log(n\tau)\cdot\log n)$ arithmetic operations over the rational numbers.
\end{theorem}

Notice that, from the latter theorem, we can immediately derive a corresponding result for polynomials $p(x)=\sum_{i=0}^{k-1}a_{i_l}x^{i_l}\in\Q[x]$ with rational 
coefficients. Namely, suppose that $a_{i_1}=\frac{p_{i_l}}{q_{i_l}}$, 
with integers $p_{i_l}$ and $q_{i_l}$ of absolute values less than $2^\tau$. Then, $P(x):=p(x)\cdot \prod_{l=0}^{k-1}q_{i_l}\in\Z[x]$ is a $k$-nomial of magnitude $(n,k\tau)$ that has the same roots as $p(x)$. Since $P$ can be compute from $p$ using less than $2k$ multiplications, we conclude from Theorem~\ref{thm:main1}:

\begin{corollary}
Let $p(x)=\sum_{l=0}^{k-1}a_{i_l}x^{i_l}\in\Q[x]$ be a polynomial with rational coefficients of the form $a_{i_l}=\frac{p_{i_l}}{q_{i_l}}$, where $p_{i_l},q_{i_l}\in\Z$ and $|p_{i_l}|,|q_{i_l}|<2^\tau$ for all $l=0,\ldots,k-1$. Then, all real roots of $p$ can be isolated with $O(k^3\cdot \log(k n\tau)\cdot\log n)$ arithmetic operations over the rational numbers.
\end{corollary}

\section{Root Refinement}\label{sec:refinement}

Throughout the following considerations, let $p(x)$ be a $k$-nomial of magnitude $(n,\tau)$ as in (\ref{polyf}), and let $I_0=(a_0,b_0)\subset\mathcal{I}=(0,2^{\tau+1})$, with $a_0,b_0\in\Q$, be an isolating interval for a simple real root 
$\xi$ of~$p$. For a given positive integer $L\in\Z$, we aim to refine $I_0$ to a width less than $2^{-L}$. Our refinement 
method is almost identical to a combination of the \emph{Newton-} and the \emph{Boundary-Test} as proposed in a very 
recent paper~\cite{DBLP:journals/corr/SagraloffM13} on real root isolation, however, we slightly modify the latter approach in order to exploit the sparsity of $p$. That is, for testing an interval $I\subset I_0$ for the 
existence of a root, we replace a test based on Descartes' Rule of Signs (see~Theorem~\ref{propertiesvar}) by a simple sign evaluation of $p$ at the endpoints of 
$I$. For refining $I_0$, this is possible as $I_0$ is assumed to be isolating for a simple root, whereas the method from~\cite{DBLP:journals/corr/SagraloffM13} has to process arbitrary intervals for which no further information is provided.\footnote{\small The Newton-Test from~\cite[Section 3.2]{DBLP:journals/corr/SagraloffM13} is a crucial subroutine within the root isolation algorithm \textsc{ANewDsc}. It guarantees that, during the isolation process, clusters of roots are automatically detected and further approximated in an efficient manner. In this setting, the Newton-Test applies to arbitrary intervals $I$ that are not known to be isolating yet. Notice that, for the refinement of $I_0$, we cannot directly use the original Newton-Test from~\cite{DBLP:journals/corr/SagraloffM13}. Namely, in general, the polynomial $p_I$ from (\ref{polypI}) is not sparse anymore, even for small~$k$, and thus, we would need a super linear number of arithmetic operations to compute $\var(p,I)$ (see Thm.~\ref{propertiesvar} for definitions). However, when refining an interval $I_{0}$ that is known to isolate a simple real root $\xi$ of $p$, we can test a subinterval $I\subset I_{0}$ for being isolating with only two evaluations of $p$. Namely, $I$ isolates $\xi$ if and only if $p(a)\cdot p(b)<0$.} 

For the sake of a self-contained presentation, we briefly review some basic facts about Descartes' Rule of Signs before presenting the refinement algorithm; for an extensive treatment of the Descartes method, we refer to~\cite{Collins-Akritas,eigenwillig-phd,ESY06,Sagraloff12,DBLP:journals/corr/SagraloffM13}:

\begin{figure}[t]
\begin{center}
\includegraphics[width=4.5cm]{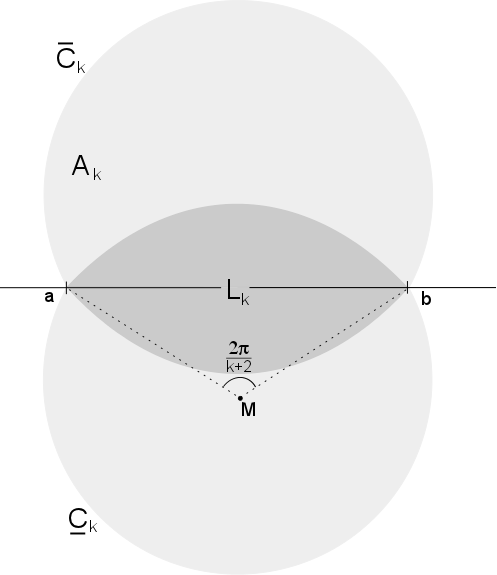}\end{center}
\caption{\label{fig:Obreshkoff} For an arbitrary integer $k\in\{0,\ldots,n\}$, let 
$\overline{C}_k$ and $\underline{C}_k$ for $I:=(a,b)$ have the endpoints of $I$ on their
boundaries; their centers see the line segment $\overline{ab}$ under the angle $\frac{2\pi}{k+2}$.  
The \emph{Obreshkoff lens} $L_k$ is the interior of $\overline{C} \cap
\underline{C}_k$, and the \emph{Obreshkoff area} $A_k$ is the interior of $\overline{C}_k \cup
\underline{C}_k$. $A_0$ and $A_1$ are called the \emph{One-} and \emph{Two-Circle Regions} of $I$, respectively.\vspace{-0.25cm}}
\end{figure}

For an arbitrary interval $I=(a,b)$, we denote $\var(p,I)$ the number of
sign variations in the coefficient sequence $(a_{I,0},\ldots,a_{I,n})$ (after removing all zero-entries) of the polynomial
\begin{align}\label{polypI}
p_I(x)=\sum_{i=0}^n a_{I,i} x^i:=(x+1)^n\cdot f\left(\frac{a\cdot x+b}{x+1}\right).
\end{align}
The polynomial $p_I$ is computed from $p$ via the M\"obius transformation that maps a point $x\in\C\backslash \{-1\}$ to $\frac{a\cdot x+b}{x+1}\in\C$, followed by multiplication with $(x+1)^n$. Notice that the latter step ensures that denominators in $p((ax+b)/(x+1))$ are cleared. There is a one-to-one correspondence (preserving multiplicities) between the positive real roots of $p_I$ and the roots of $p$ in $I$. In addition, according to Descartes' Rule of Signs, $v:=\var(p,I)$ is an upper bound on the number $m$ of real roots of $p$ in $I$ and $v-m$ is an even integer. 
The function $\var(p,\cdot)$ has several further important properties:

\begin{theorem}\label{propertiesvar}\cite{eigenwillig-phd,Obreschkoff:book,Obrechkoff:book-english}\quad
Let $I=(a,b)$ be an arbitrary interval, and let $L_i$ and $A_i$, with $i=0,1,\ldots,n$, be the Obreshkoff regions in $\C$ as defined in Figure~\ref{fig:Obreshkoff}. Then, it holds that (roots are counted with multiplicity):
\begin{itemize}
\item[(a)]
\# roots contained in $L_n$ $\le$ $\var(p,I)$ $\le$ \# roots contained in~$A_n$
\item[({b})] If $p$ contains no root in $A_0$, then $\var(p,I)=0$. If $A_1$ contains exactly one root, then $\var(p,I)=1$.
\item[({c})] If $I_1$ and $I_2$ are two disjoint subintervals of $I$, then
\[
\var(p,I_1) + var(p,I_2) \le \var(p,I).
\]
\end{itemize}
\end{theorem}

From Theorem~\ref{propertiesvar} (c), we conclude that, for any interval $I=(a,b)\subset \R_{>0}$ on the positive real axis, $\var(p,I)$ is upper bounded by $\var(p,(0,b))=\var(p)$, and thus, $\var(p,I)\le k-1$. In particular, we have $\var(p,I_0)\le k-1$. Hence,  part (a) of Theorem~\ref{propertiesvar} implies that the Obreshkoff lens $L_n$ of $I_0$ contains at most $k-1$ roots of $p$.

We can now formulate our refinement method. As mentioned above, it is almost identical to the approach presented in~\cite[Section 3.2]{DBLP:journals/corr/SagraloffM13}, hence we keep the presentation at hand as short as possible and refer to the corresponding paper for more explanations and
for a discussion that motivates the approach:

The main idea is to iteratively refine $I_0$ such that, in each iteration, we replace an isolating interval $I=(a,b)\subset I_0$ by an isolating interval $I'=(a',b')\subset I$ of considerably smaller width, and with $a',b'\in\Q$. For this, we use two main ingredients, namely, sign evaluation of $p$ at the endpoints of $I'$ in order to test $I'$ for the existence of a root, and a subdivision strategy based on Newton iteration and bisection, which guarantees quadratic convergence in the majority of all step. We give details:\medskip 

\noindent\textbf{Algorithm} \textsc{NewRefine} (read Newton-Refine)\medskip

\noindent Input: An interval $I_0=(a_0,b_0)\subset \mathcal{I}=(0,2^{\tau+1})$, with endpoints $a_0,b_0\in\Q$, that isolates a simple root $\xi\in\R$ of a polynomial $p\in\Z[x]$, and a positive integer $L$.\\
Output: An interval $I=(a,b)\subset I_0$, with endpoints $a,b\in\Q$, of width less than $2^{-L}$ that isolates $\xi$.\medskip
 
In each step of the recursion, we store a pair $\mathcal{A}:=(I,N_I)$, where we initially set $\mathcal{A}:=(I_0,N_{I_0})$, with $N_{I_0}:=4$. We first try to compute a subinterval $I'\subset I$ of width $w(I')$, with $\frac{w(I)}{8N_I}\le w(I')\le \frac{w(I)}{N_I}$, that contains the unique root $\xi$. This is done via two tests, that is, the \texttt{Newton-Test\_signat} and the \texttt{Boundary-Test\_signat}.\footnote{\small In order to distinguish the tests from their counterparts in~\cite{DBLP:journals/corr/SagraloffM13}, we use the affix \texttt{\_signat}, which refers to the sign evaluation of $p$ at the endpoints of an interval $I$ in order to test $I$ for the existence of a root. We also remark that we directly give both tests in its full generality. That is, at several places, we use approximate arithmetic, and, in addition, we allow to choose an arbitrary point $m_i$ from $\multipoint{m}{\delta}$, where $\multipoint{m}{\delta}:=\set{m_i:=m+(i-\lceil k/2\rceil)\cdot \delta}{i=0,\ldots,2\cdot \lceil k/2\rceil}$ is a set of $2\cdot\lceil k/2\rceil +1$ points that are clustered at $m$. For now, the reader may assume that we always choose $m_i=m$, and that exact arithmetic over rational numbers is used. However, for our variant of the root isolation algorithm that uses approximate arithmetic (see Section~\ref{bitcomplexity}), we will exploit the fact that we can choose a point $m_i$ from $\multipoint{m}{\delta}$ for which $|p(m_i)|$ becomes large. This guarantees that the precision does not become unnecessarily large.\label{footnote:apx}}
\medskip   

\hrule\nopagebreak\medskip
\noindent \texttt{Newton-Test\_signat}:
Consider the points $\xi_{1}:=a+\frac{1}{4}\cdot w(I)$, $\xi_{2}:=a+\frac{1}{2}\cdot w(I)$, $\xi_{3}:=a+\frac{3}{4}\cdot w(I)$, and let $\epsilon := 2^{-\ceil{5 + \log n}}$.
For $j = 1,2,3$, we choose arbitrary points (see Footnote~\ref{footnote:apx})
\begin{equation}\label{Newtonmultipoint1}
\xi_{j}^{*}\in \multipoint{\xi_{j}}{\epsilon \cdot w(I)},
\end{equation} 
where, for an arbitrary $m\in\R$ and an arbitrary $\delta\in\R_{>0}$, we define 
$$\multipoint{m}{\delta}:=\set{m_i:=m+(i-\lceil k/2\rceil)\cdot \delta}{i=0,\ldots,2\cdot \lceil k/2\rceil}.$$
The points $\xi_j^*$ define values
$
v_{j}:=\frac{p(\xi_{j}^{*})}{p'(\xi_{j}^{*})}.
$
Then, for the three distinct pairs of indices $i,j\in\{1,2,3\}$ with $i<j$, we perform the following computations in parallel: 
For $L=1,2,4,\ldots$, we  compute 
approximations of $p(\xi^{*}_{i})$, $p(\xi^{*}_{j})$, $p'(\xi^{*}_{i})$, and 
$p'(\xi^{*}_{j})$ to $L$ bits after the binary point; see Footnote~\ref{footnote:apx}. We stop doubling $L$ for a particular pair $(i,j)$ if we can 
either verify that 
\begin{align}\label{condition1}
|v_{i}|,|v_{j}|>w(I)\quad\text{or}\quad |v_{i}-v_{j}|<\frac{w(I)}{4n}\vspace{-0.2cm}
\end{align} 
or that\vspace{-0.2cm}
\begin{align}\label{condition2}
|v_{i}|,|v_{j}|<2\cdot w(I)\quad\text{and}\quad |v_{i}-v_{j}|>\frac{w(I)}{8n}.
\end{align} 

If (\ref{condition1}) holds, we discard the pair $(i,j)$. Otherwise, we compute sufficiently good
approximations (see Footnote~\ref{footnote:apx}) of the values $p(\xi^{*}_{i})$, $p(\xi^{*}_{j})$, $p'(\xi^{*}_{i})$, and $p'(\xi^{*}_{j})$, such that we can derive an approximation $\tilde{\lambda}_{i,j}$ of\vspace{-0.25cm}
\begin{align}\label{def:lambda}
\lambda_{i,j}:=\xi_{i}^{*}+\frac{\xi^{*}_{j}-\xi^{*}_{i}}{v_{j}-v_{j}}\cdot v_{i}
\end{align}
with $|\tilde{\lambda}_{i,j}-\lambda_{i,j}|\le \frac{1}{32N_{I}}$.
If $\tilde{\lambda}_{i,j} \not\in[a,b]$, we discard the pair $(i,j)$. Otherwise, let $\ell_{i,j}:=\floor{\frac{4N_I\cdot(\tilde{\lambda}_{i,j}-a)}{w(I)}}$. Then, it holds that $\ell_{i,j} \in  \{0,\ldots,4N_{I}\}$. We further define
\begin{align*}
I_{i,j}&:=(a_{i,j},b_{i,j})\\
:=&\left(a+\max(0,\ell_{i,j}-1)\cdot\frac{w(I)}{4N_{I}},a+\min(4N_{I},\ell_{i,j}+2)\cdot\frac{w(I)}{4N_{I}}\right).
\end{align*}
If $a_{i,j}=a$, we set $a_{i,j}^{*}:=a$, and if $b_{i,j}=b$, we set $b_{i,j}^*:=b$. For all other values for $a_{i,j}$ and $b_{i,j}$, we choose arbitrary points (see Footnote~\ref{footnote:apx})
\begin{equation}\label{Newtonmultipoint2}
a_{i,j}^{*}\in \multipoint{a_{i,j}}{\epsilon\cdot \frac{w(I)}{N_I}}
\quad\text{and}\quad
b_{i,j}^{*}\in \multipoint{b_{i,j}}{\epsilon\cdot \frac{w(I)}{N_I}}.
\end{equation}
We define $I':=I_{i,j}^{*}:=(a_{i,j}^{*},b_{i,j}^{*})$. Notice that $I'$ is contained in $I$, and it holds that $\frac{w(I)}{8N_{I}}\le w(I')\le \frac{w(I)}{N_{I}}$. In addition, if the endpoints of $I$ are dyadic, then the endpoints of~$I'$ are dyadic as well. 

In the final step, we compute the sign of $p(a_{i,j}^{*})$ and $p(b_{i,j}^{*})$. $I'$ is isolating for $\xi$ if and only if $p(a_{i,j}^{*})\cdot p(b_{i,j}^{*})<0$, hence, if the latter inequality is fulfilled, we return $I'$. Otherwise, we discard $(i,j)$.

We say that the \texttt{Newton-Test\_signat} succeeds if it returns an interval $I'=I_{i,j}^{*}$ for at least one of the three pairs $(i,j)$. If we obtain an interval for more than one pair, we can output either one of them. Otherwise, the test fails.

\medskip
\hrule\medskip

If the \texttt{Newton-Test\_signat} succeeds, we replace $\mathcal{A}=(I,N_I)$ by $\mathcal{A}:=(I',N_{I'})$, with $N_{I'}:=N_I^2$. If the \texttt{Newton-Test\_signat} fails, we continue with the so-called \texttt{Boundary-Test\_signat}. Essentially, it checks whether $\xi$ is located very close to one of the endpoints of $I$.\medskip

\hrule \nopagebreak \medskip
\noindent\texttt{Boundary-Test\_signat:}
Let $m_{\ell}:=a+\frac{w(I)}{2N_{I}}$ and $m_{r}:=b-\frac{w(I)}{2N_{I}}$, and let $\epsilon := 2^{-\ceil{2 + \log n}}$. Choose arbitrary points (see Footnote~\ref{footnote:apx}) 
\begin{equation}\label{Boundarymultipoint}
m_{\ell}^{*}\in  \multipoint{m_{\ell}}{\epsilon \cdot\frac{w(I)}{N_I}} \quad\text{and}\quad m_{r}^{*}\in \multipoint{m_{r}}{\epsilon\cdot\frac{w(I)}{N_I}}, 
\end{equation}
and compute the sign of $p(x)$ at $x=a$, $x=m_{\ell}^{*}$, $x=m_{r}^{*}$, and $x=b$.
If $p(a)\cdot p(m_{\ell}^{*})<0$, then $I':=(a,m_{\ell})$ isolates $\xi$, and thus, we return $I'$. If $p(b)\cdot p(m_{r}^{*})<0$, we return $I'=(m_{r},b)$. Notice that from our definition of $m_{\ell}^{*}$ and $m_{r}^{*}$, it follows that both intervals $I_{\ell}$ and $I_{r}$ have width in between $\frac{w(I)}{4N_{I}}$ and $\frac{w(I)}{N_{I}}$. If $p(a)\cdot p(m_{\ell}^{*})<0$ or $p(b)\cdot p(m_{r}^{*})<0$, the \texttt{Newton-Test\_signat} succeeds. Otherwise, the test fails. \nopagebreak
\medskip
\hrule\medskip

If the \texttt{Boundary-Test\_signat} succeeds, then $\mathcal{A}=(I,N_I)$ is replaced by $\mathcal{A}:=(I',N_{I'})$, with $N':=N_I^2$.
If the \texttt{Newton-Test\_signat} as well as the \texttt{Boundary-Test\_signat} fail, then we choose an arbitrary point (see Footnote~\ref{footnote:apx})
\begin{align}\label{Bisectionmultipoint}
m^* \in \multipoint{m(I)}{\frac{w(I)}{2^{\ceil{2 + \log n}}}}
\end{align}
and compute the sign of $p(x)$ at $x=a$ and $x=m^*$. If $p(a)\cdot p(m^*)<0$, we replace $\mathcal{A}=(I,N_I)$ by $\mathcal{A}:=(I',N_{I'})$, with $I'=(a,m^*)$ and $N_{I'}:=\max(4,\sqrt{N_{I}})$. If $p(m^*)=0$, we stop and return the interval $[m^*]$ of width zero. Otherwise, we replace $\mathcal{A}=(I,N_I)$ by $\mathcal{A}:=(I',N_{I'})$, with $I'=(m^*,b)$ and $N_{I'}:=\max(4,\sqrt{N_{I}})$. We stop refining $I$ as soon as $I$ has width less than~$2^{-L}$.\bigskip

We formulated the \texttt{Newton-Test\_signat} and the \texttt{Boundary-Test\_signat} in a way such that each of them 
succeeds if the corresponding test in~\cite{DBLP:journals/corr/SagraloffM13} succeeds, assuming that we choose the same 
points in~(\ref{Newtonmultipoint1}),~(\ref{Newtonmultipoint2}), and in~(\ref{Boundarymultipoint}). Namely, if $I'=(a',b')$ is 
known to contain at most one (simple) root of $p$, then $\var(p,I')=0$ implies that $p(a')\cdot p(b')\ge 0$.\footnote{\small $p(a')\cdot p(b')\ge 0$ implies that $I'$ contains no root but not that $\var(p,I')=0$.} Hence, 
the analysis from~\cite{DBLP:journals/corr/SagraloffM13} directly carries over and yields the following result:\footnote{\small The proof of Lemma~\ref{bound1:connection} is essentially identical to our considerations 
in~\cite[Section 3.2 and 4.1]{DBLP:journals/corr/SagraloffM13}. In particular, the proofs 
of~\cite[Lemma 20 and 23]{DBLP:journals/corr/SagraloffM13} directly carry over if we use that $\log N_I$ is always bounded by $O(\tau+L)$ when refining $I_0$ to a width less 
than~$2^{-L}$.}

\begin{lemma}\label{bound1:connection}
Let $I_0,I_1,\ldots,I_s$, with $I_0\supset I_1\supset\cdots\supset I_s$, $s\in\N$, and $w(I_{s-1})\ge 2^{-L}>w(I_s)$, be the intervals produced by the algorithm \textsc{NewRefine}, and let $s_{\max}$ be the largest number of intervals $I_j$ for which the one-circle region of $I_j$ contains exactly the same roots. Then, it holds that
$
s_{\max}=O(\log n+\log(\tau+L)).
$
\end{lemma}

From the lemma above and Theorem~\ref{propertiesvar}, we now obtain the following bound for the number of iterations that is needed to refine $I_0$ to an interval of width less than $2^{-L}$. 

\begin{theorem}\label{thm:arithmeticcomplexity}
Let $I_0=(a_0,b_0)\subset (0,2^{\tau+1})$, with $a_0,b_0\in\Q$, be an isolating interval for a simple root $\xi$ of a $k$-nomial $p\in\Z[x]$ of magnitude $(n,\tau)$. For computing an interval $I=(a,b)\subset I_0$, with $a,b\in\Q$ and $\xi\in I$, the algorithm \textsc{NewRefine} needs
$
O(\var(p,I_0)\cdot (\log n+\log(\tau+L)))
$
many iterations and $O(k\cdot\log n\cdot \var(p,I_0)\cdot (\log n+\log(\tau+L))$ many arithmetic operations over $\Q$.
\end{theorem}

\begin{proof}
As in Lemma~\ref{bound1:connection}, we denote $I_0,I_1,\ldots,I_s$, with $I_0\supset I_1\supset\cdots\supset I_s$, $s\in\N$, and $w(I_{s-1})\ge 2^{-L}>w(I_s)$, the intervals produced by \textsc{NewRefine}. 
Let $j_0$ be the minimal index $j$ for which the one-circle region $A_0$ of $I_j$ contains at most $v_0$ many roots, with $v_0:=\var(p,I_0)$. If the one-circle region of each $I_j$ contains more than $v_0$ roots, we set $j_0:=s$.
Now, using Lemma~\ref{bound1:connection} for the sub-sequence $I_{j_0},\ldots,I_{s}$, we conclude that $s-j_0=v_0\cdot s_{\max}$. 
Hence, we are left to argue that $j_0$ is bounded by $O(v_0\cdot (\log n+\log(\tau+L)))$. In fact, the following consideration even shows that $j_0=O(\log n+\log(\tau+L))$: We first consider the special case, where $I_{j_0}$ shares a common endpoint with $I_0$. Then, exactly the same argument as in the proof of~\cite[Lemma 23]{DBLP:journals/corr/SagraloffM13} shows that $j_0$ is bounded by $O(\log n+\log(\tau+L))$, where we use that $N_{I_j}=O(L+\tau)$ for all $j$. Essentially, this is due to the fact that success of the \texttt{Boundary-Test\_signat} guarantees quadratic convergence, and the latter test must succeed for all but $O(\log n+\log(\tau+L))$ many iterations. Now suppose that there exists an index $j_0'<j_0$ such that $I_{j_0'}$ shares a common endpoint with $I_0$, whereas $I_{j_0'+1}$ does not. Then, $j_0'=O(\log n+\log(\tau+L))$. In addition, the distance from any point $x\in I_{j_0'+1}$ to each of the two endpoints $a_0$ and $b_0$ is larger than or equal to $w(I_{j_0'+1})/4$. Hence, since $w(I_{j+1})\le \frac{3}{4}\cdot w(I_{j})$ for all $j$, we have $\max(|a_j-a_0|,|b_j-b_0|)>8n^2\cdot w(I_j)$ for all $j>j_0'+4(\log n+1)$. According to~\cite[Lemma 9]{Sagraloff12}, it follows that the one-circle region of any interval $I_j$, with $j>j_0'+4(\log n+1)$, is contained in the Obreshkoff lens $L_n$ of $I_{0}$. Now, from part (a) of Lemma~\ref{propertiesvar}, we conclude that the one-circle region of $I_j$ contains at most $v_0$ roots for each $j>j_0'+4(\log n+1)$, and thus, $j_0=O(\log n+\log(\tau+L))$. This proves the first claim.

For the second claim, we remark that, in each iteration, we perform a constant number of evaluations of the polynomial $p$ and its derivative $p'$. Since both polynomials have $k$ coefficients or less, this shows that we need $O(k\log n)$ arithmetic operations over $\Q$ in each iteration. Multiplication of the latter bound with the bound on the number of iterations eventually yields the claimed bound on the arithmetic complexity.
\end{proof}

Now suppose that isolating intervals $I_1,\ldots,I_{k_0}\subset \mathcal{I}$ for all simple real roots of $p$ are given. Then, $\sum_{j=1}^{k_0}\var(p,I_j)\le \var(p,\mathcal{I})\le k$, and thus, Theorem~\ref{thm:arithmeticcomplexity} yields the following result: 
\begin{corollary}\label{cor:arithmeticcomplexity}
Let $p\in\Z[x]$ be a $k$-nomial of magnitude $(n,\tau)$, and $I_j=(a_j,b_j)\subset \mathcal{I}=(0,2^{\tau+1})$, with $j=1,\ldots,k_0$ and $a_j,b_j\in\Q$, be isolating intervals for all simple real roots of $p$. Then, we can refine all intervals $I_j$ to a width less than $2^{-L}$, with $L$ an arbitrary positive integer, with a number of arithmetic operations over $\Q$ bounded by 
$
O(k^2\cdot \log n\cdot (\log n+ \log(\tau+L))).
$
\end{corollary}

\section{Bit Complexity}\label{bitcomplexity}

We finally aim to derive a bound on the bit complexity of our algorithm when using approximate but certified arithmetic. When using exact arithmetic over dyadic numbers (i.e.~numbers of the form $m\cdot 2^{-l}$, 
with $m,l\in\Z$), all intermediate results are dyadic and of bit-size $O(n^2(\log n+\tau))$. 
Namely, we refine intervals to a width of size $2^{-O(n(\tau+\log n))}$, and only consider evaluations of the polynomial $p$ 
at dyadic points that are contained in such intervals and whose bit-size is bounded by $\kappa=O(n(\tau+\log n))$. From the latter fact and Theorem~\ref{thm:main1}, we conclude that the overall bit complexity of our algorithm 
is bounded by $\tilde{O}(n^2\tau\cdot k^3)$ when using exact arithmetic over rational (dyadic) numbers. Here, we use that exact evaluation of $p$ at a dyadic number of bit-size $\kappa$ needs $\tilde{O}(n(\kappa+\tau))$ bit operations~\cite[Lemma~2]{DBLP:conf/issac/BouzidiLPR13}.
However, the following considerations show that we can replace a factor $n$ by an additional factor $k$ in the latter bound. More precisely, using approximate computation, we can reduce the bit-size of the intermediate results by a factor $n$ for the price of using $k$ times as many arithmetic operations. We give details:\\

Notice that, at several places in the algorithm $\textsc{NewRefine}$, that is, in (\ref{Newtonmultipoint1}),~(\ref{Newtonmultipoint2}),~(\ref{Boundarymultipoint}), and in (\ref{Bisectionmultipoint}), we are free to choose an arbitrary point $m_i$ from a set $$\multipoint{m}{\delta}:=\set{m_i:=m+(i-\lceil k/2\rceil)\cdot \delta}{i=0,\ldots,2\cdot \lceil k/2\rceil}$$ 
consisting of $\lceil k/2\rceil +1$ points that are clustered at $m$. Now, in order to keep the precision of the computations as low as possible, we aim to choose a point $m_i\in\multipoint{m}{\delta}$ for which $p(m_i)$ has a large absolute value. We introduce the following definition that has already been used in~\cite[Section 2.2]{DBLP:journals/corr/SagraloffM13} in a slightly modified form.

\begin{definition}\label{admissible point} For $\multipoint{m}{\delta}$ as above, we call a point $m^*\in \multipoint{m}{\delta}$ \emph{admissible with respect to $\multipoint{m}{\delta}$} (or just \emph{admissible} if there is no ambiguity) if $|p(m^*)|\ge \frac{1}{4}\cdot\max_{i}|p(m_{i})|$.
\end{definition}

\begin{lemma}\label{lem:apxmultipointeval}
Suppose that each point in $\multipoint{m}{\delta}$ has absolute value less than $2^{\tau+1}$ and that $\lambda:=\max_{i}|p(x_{i})|\neq 0$. Then, we can determine an admissible point $m^*\in \multipoint{m}{\delta}$ and an integer $t$ with 
\[
2^{t-1}\le |p(m^{*})|\le \lambda \le 2^{t+1} 
\]
with $\tilde{O}(k(n\tau+\log \max(\lambda^{-1},1)))$ many bit operations. 
\end{lemma}

\begin{proof}
Using the same approach as in~\cite[Section 4]{qir-kerber-11} plus repeated squaring, we can evaluate $p$ at any of the $\lceil k/2\rceil +1$ many points $x=m_i$ to an absolute error less than $2^{-K}$, with $K$ an arbitrary positive integer, in a number of bit operations bounded by $\tilde{O}(k\cdot (n\tau+K))$. 
We can now compute an admissible point $m^*\in\multipoint{m}{\delta}$ as follows: 
Consider $K=1,2,4,8,\ldots$ and approximate \emph{all values} $|p(m_{i})|$ to a precision of $K$ bits after the binary point until, for at least one $i$, we obtain an approximation $2^{t_{i}}$ with $t_{i}\in\Z$ and $2^{t_{i}-1}\le |p(x_{i})|\le 2^{t_{i}+1}$. Now, let $i_{0}$ be such that $t_{i_{0}}$ is maximal; then, it follows that $2^{t_{i_{0}}-1}\le \lambda\le 2^{t_{i_{0}}+1}$. Following this approach, we must stop for a $K$ with $K<2\log \max(\lambda^{-1},1)$. Since we double $K$ at most $\log\log \max(\lambda^{-1},1)$ many times, the claim follows.
\end{proof}

Now, in~(\ref{Newtonmultipoint1}),~(\ref{Newtonmultipoint2}),~(\ref{Boundarymultipoint}), and in (\ref{Bisectionmultipoint}) of \textsc{NewRefine}, we do not choose an arbitrary point from the corresponding set $\multipoint{m}{\delta}$ but an admissible point $m^*\in \multipoint{m}{\delta}$. We argue that, for each such $m^*$, we have $|p(m^*)|>2^{-O(n(\tau+\log n))}$: Let $I=(a,b)$ be the interval that is processed in the current iteration, then
$\min(|m^*-a|,|m^*-b|)>n\delta\ge\left(\sin\frac{\pi}{2n+2}\right)^{-1}\cdot\frac{\delta}{2}.$
Hence, the distance from $m^*$ to the boundary of the Obreshkoff lens $L_n$ of $I$ is larger than $\delta/2$ since the distance from an arbitrary point $x\in I=(a,b)$ to the boundary of $L_n$ is larger than $\min(|x-a|,|x-b|)\cdot \sin\frac{\pi}{2n+2}$; see~\cite[Lemma 5 and Figure 4.1]{Sagraloff12}. Since the Obreshkoff lens $L_n$ of the larger interval $I_0$ contains at most $k$ roots (counted with multiplicity), it follows that there exists at least one point $m_{i_0}\in \multipoint{m}{\delta}$ with $|\xi_j-m_{i_0}|\ge \delta/2$ for all (distinct) complex roots $\xi_j$ of $p$. Let $\xi_{j_0}$ be the root of $p$ that minimizes the distance to $m_{i_0}$. If $\xi_{j_0}=\xi$, then 
\[
\frac{|p(m_{i_0})|}{|p'(\xi)|}=|m_{i_0}-\xi|\cdot\prod_{i\neq j_0} \left(\frac{|m_{i_0}-\xi_j|}{|\xi-\xi_j|}\right)^{\mu_j}\ge |m_{i_0}-\xi|\cdot 2^{-n+1}\ge \frac{\delta}{2^{n}},
\]
where $\mu_j$ denotes the multiplicity of $\xi_j$ as a root of $p$. Hence, from $\delta\ge 2^{-\lceil 5+\log n\rceil}\cdot\frac{w(I)}{N_I}=2^{-O(\log n+\tau+L)}$, we conclude that $|p(m_{i_0})|=2^{-O(n(\log n+\tau))}$ if $w(I)\ge 2^{-L}$, with $L:=128n\cdot(\log n+\tau)$. We are left to discuss the case $\xi_{j_0}\neq \xi$. Then,
\begin{align*}
\frac{|p(m_{i_0})|}{|p^{(\mu_{j_0})}(\xi_{j_0})|}&=|m_{i_0}-\xi_{j_0}|^{\mu_{j_0}}\cdot\prod_{i\neq j_0} \left(\frac{|m_{i_0}-\xi_j|}{|\xi_{j_0}-\xi_j|}\right)^{m_j}\ge \frac{|m_{i_0}-\xi_{j_0}|^{\mu_{j_0}}}{2^{n-1}}\\
&\ge 2^{-2n}\cdot \delta^{\mu_{j_0}}\ge 2^{-2n-n(6+\log n)}\cdot \left(\frac{w(I)}{N_I}\right)^{\mu_{j_0}}.
\end{align*}
Trivially, we have $w(I)\le 2\cdot w(I)/\sqrt{N_I}$ for $N_I=4$. If $N_I>4$, then there must have been an iteration, where we replaced an isolating interval $J$ for $\xi$, with $I\subset J\subset (0,2^{\tau+1})$ by an interval $J'$, with $I\subset J'\subset J$ and $w(J')\le w(J)/\sqrt{N_I}$. Hence, in any case, we have $w(I)\le 2^{\tau+2}/\sqrt{N_I}$. This shows that 
\begin{align*}
\left(\frac{w(I)}{N_I}\right)^{\mu_{j_0}}&\ge w(I)^{3\mu_{j_0}}\cdot 2^{-2\mu_{j_0}\cdot(\tau+2)}\ge |\xi-\xi_{j_0}|^{3\mu_{j_0}}\cdot 2^{-2n(\tau+2)},
\end{align*}
where the second to last inequality follows from the inequality $|\xi-\xi_{j_0}|\le |\xi-m_{i_0}|+|m_{i_0}-\xi_{j_0}|\le w(I)+|m_{i_0}-\xi_{j_0}|$. Then, Theorem~\ref{evalbound} (with $F:=p\cdot 1$) implies that $\left(\frac{w(I)}{N_I}\right)^{\mu_{j_0}}=2^{-O(n(\log n+\tau))}$.
In summary, we conclude that, in~(\ref{Newtonmultipoint1}),~(\ref{Newtonmultipoint2}),~(\ref{Boundarymultipoint}), and in (\ref{Bisectionmultipoint}), we can choose points $m_i\in \multipoint{m}{\delta}$ with $|p(m_i)|=2^{-O(n(\log n+\tau))}$
for the cost of $\tilde{O}(k\cdot n\tau)$ bit operations. Notice that, for the same cost, we can also determine the sign of $p$ at each of these points, and thus, the considered sign evaluations in one iteration need $\tilde{O}(k\cdot n\tau)$ bit operations.

It remains to bound the cost for computing the approximations $\tilde{\lambda}_{j_1,j_2}$ as defined in (\ref{def:lambda}) in \textsc{NewRefine}. Notice that, for checking the inequalities in (\ref{condition1}) and in (\ref{condition2}), it suffices to approximate the values $p(\xi^{*}_{j_{1}})$, $p(\xi^{*}_{j_{2}})$, $p'(\xi^{*}_{j_{1}})$, and 
$p'(\xi^{*}_{j_{2}})$ to an absolute error of $$\log \min(p(\xi^{*}_{j_{1}}),p(\xi^{*}_{j_{2}}))+O(\log (n/w(I)))=O(n(\log n+\tau))$$ bits after the binary point. Again, the cost for computing such approximations is bounded by $\tilde{O}(k\cdot n\tau)$ bit operations. Then, the same complexity bounds also holds for the computation of $\tilde{\lambda}_{j_1,j_2}$. Namely, since $v_{j_1}-v_{j_2}$ has absolute value larger than $w(I)/(8n)$, and $v_{j_1}$ as well as $v_{j_2}$ have absolute value smaller than $2^{O(n\tau)}$, it suffices to carry out all operations with a precision of $O(\log N_I+\log w(I)^{-1}+n\tau)$ bits after the binary point. We summarize:

\begin{lemma}\label{refinement:bitcomplexity}
Let $p\in\Z[x]$ be a $k$-nomial of magnitude $(n,\tau)$, and let $I_j=(a_j,b_j)\subset \mathcal{I}=2^{\tau+1}$, with $j=1,\ldots,k_0$, be isolating intervals for all simple real roots of $p$. Suppose that $a_j,b_j\in\Q$ and $\min_{j}\min(|p(a_j)|,|p(b_j)|)>2^{-L}$, with $L:=128n\cdot(\log n+\tau)$. Then, \textsc{NewRefine} refines all intervals $I_j$ to a width less than $2^{-L}$ with a number of bit operations bounded by 
$
\tilde{O}(k^3\cdot n\tau).$
For each interval $I_j'=(a_j',b_j')$ returned by \textsc{NewRefine}, we have $a_j',b_j'\in\Q$ and $\min(|p(a_j')|,|p(b_j')|>2^{-L}$.
\end{lemma}

\begin{proof}
The result is an almost immediate consequence of our considerations above. Notice that the condition on the endpoints of the initial intervals $I_j$ guarantees that we only evaluate the sign of $p$ at points that are either admissible points $m^*\in\multipoint{m}{\delta}$ or endpoints of one of the intervals $I_j$. Hence, each such sign evaluation needs $\tilde{O}(k\cdot n\tau)$ bit operation. For the computation of an admissible point, we need $\tilde{O}(k^2\cdot n\tau)$ many bit operations, which is due to the fact that we perform approximate computation of $p$ at $O(k)$ many points in parallel. From Theorem~\ref{thm:arithmeticcomplexity}, we conclude that the number of iterations in total is bounded by $O(k\cdot \log (n\tau))$, and thus, the claimed bound on the bit complexity follows. 
\end{proof}

For our root isolation algorithm as proposed in Section~\ref{sec:algorithm}, the above result implies that we can isolate all 
real roots of $p$ with a number of bit operations bounded by $\tilde{O}(k^4\cdot n(\tau+k))$. Namely, in each step of the 
recursion, we first have to evaluate some $k$-nomial $p_{j-1}$ of magnitude $(n,\tau+k\log n)$ at arbitrary points 
$x_i\in I_{j,i}$, where $I_{j,i}=(a_{j,i},b_{j,i})$ are isolating intervals for the real roots of $p_j$. Since it suffices to 
compute approximations of the values $p_{j-1}(x_i)$ to $L/2$ bits after the binary point, the cost for all evaluations is 
bounded by $\tilde{O}(k^2\cdot n\tau)$ bit operations. In a second step, we have to refine all isolating intervals $I_{j,i}'$ 
for the simple real roots of $p_{j}$ to a width less than $2^{-L}$, with $L=128n(\log n+\tau)$. Each endpoint $e$ of an 
arbitrary $I_{j,i}$ is an endpoint of one of the intervals $I_{j,i}$, that is, $e=a_{j,i}$ or $e=b_{j,i}$ for some $i$. Hence, 
by induction, it follows that $p(e)\ge 2^{-L}$. Then, from Lemma~\ref{refinement:bitcomplexity}, we conclude that refining all 
intervals to a width less than $2^{-L}$ needs $\tilde{O}(k^4\cdot n(\tau+k))$ bit operations.

\begin{theorem}\label{maintheorem2}
Let $p\in\Z[x]$ be a $k$-nomial of magnitude $(n,\tau)$. Then, computing isolating intervals with rational endpoints for all real roots of $p$ needs $\tilde{O}(k^3\cdot n\tau)$ bit operations. For $k=O(\log (n\tau)^C)$, with $C$ a fixed positive constant, the latter bound becomes $\tilde{O}(n\tau)$, which is optimal up to logarithmic factors. 
\end{theorem} 

\begin{proof}
It remains to prove the last claim. For this, consider the polynomial 
$
p(x)=x^n-(2^{2\tau}\cdot x^2-a)^2,
$
where $a>1$ is a fixed constant integer, and $n,\tau\in\N_{\ge 8}$. Then, $p$ is a $4$-nomial of magnitude $(n,O(\tau))$, and $p$ has two positive roots $x_1$ and $x_2$, with $x_1<x_2$ and $|x_i-\sqrt{a}\cdot 2^{-\tau}|<2^{-\Omega(n\tau)}$ for $i=1,2$. Namely, let $f(x):=(2^{2\tau}\cdot x^2-a)^2$ be a polynomial that has two roots of multiplicity $2$ at $x=\pm\sqrt{a}\cdot 2^{-\tau}$, and let $g(x):=x^n$. Then, a simple computation shows that $|f(z)|>|g(z)|$ for all points $z$ on the boundary of the disk $\Delta\subset\C$ of radius $\epsilon:=2^{-(n-2)(\tau-2)}\cdot a^{n/2-1}$ centered at $\sqrt{a}\cdot 2^{-\tau}$. Hence, Rouch\'e's Theorem implies that $f$ and $p=g-f$ have the same number of roots in $\Delta$. In addition, both roots are real. 

We conclude that, for any isolating interval $I_1=(a_1,b_1)$ for $x_1$, we must have $|b_1-\sqrt{a}\cdot 2^{-\tau}|<2^{-\Omega(n\tau)}$. Now, let $b_1=p/q$ with co-prime integers $p$ and $q$, then we must have $q=\Omega(n\tau)$; see Lemma~\ref{rationalbound} in the Appendix. Hence, the binary representation of the endpoints of $I_1$ already needs $\Omega(n\tau)$ bits, which proves our claim.
\end{proof}

\section{Conclusion}

In this paper, we give the  first algorithm that computes the real roots of a sparse polynomial $p\in\Z[x]$ in a number of arithmetic operations over $\Q$ that is polynomial in the input size of the sparse representation of $p$. In addition, for sparse-enough polynomials, the algorithm achieves a near-optimal bound for the bit complexity of the problem of isolating all real roots. The main ingredients of our algorithm are evaluation and separation bounds as well as an efficient method for refining an isolating interval of a simple real root. So far, our algorithm has been formulated in the easiest possible way with the prior goal to achieve good theoretical complexity bounds, however, for the price of a probably worse efficiency in practice. Hence, our first research goal is to provide an efficient implementation of our algorithm that integrates additional steps in order to improve its practical efficiency. Our second research goal is to extend our algorithm to (square-free) polynomials with arbitrary real coefficients. For this, it seems reasonable to combine our algorithm with the root isolation method from~\cite{DBLP:journals/corr/SagraloffM13}. Hopefully, this allows us to derive improved (bit) complexity bounds for sparse polynomials that can be stated in terms of the geometry of the roots (similar to the bounds as provided in~\cite[Theorem~31]{DBLP:journals/corr/SagraloffM13} or~\cite[Theorem~3]{MSW-rootfinding2013}) rather than in terms of the input size. For polynomials $p\in\R[x]$ that may have multiple real roots, the situation becomes more complicated. Namely, since no a-priori separation bound is known to decide whether a certain root has multiplicity larger than one, we cannot distinguish between a real root of multiplicity $m>1$ and a cluster of $m$ (not necessarily real) roots. Hence, it remains an open research question whether the computation of a (reasonable good) separation bound has polynomial arithmetic complexity.

{\small
\bibliography{localref}
\bibliographystyle{abbrv}}
\newpage
\section{Appendix}

\begin{theorem}\label{appendix:evalbound}
Let $f$ and $g$ be polynomials of degree $n$ or less with integer coefficients of absolute values less than $2^{\mu}$, and let 
\begin{align}
L:=128\cdot n\cdot (\log n+\mu).
\end{align}
Then, for any two distinct roots $\xi_i$ and $\xi_j$ of $F:=f\cdot g$, it holds that $|\xi_i-\xi_j|^{m_i}>2^{-L}$, where $m_i:=\operatorname{mult}(\xi_i,F)$ denotes the multiplicity of $\xi_i$ as a root of $F$. If $\xi$ is a root of $g$ and $f(\xi)\neq 0$, then it holds that $|f(x)|>2^{-L/4}$ for all $x\in\C$ with $|x-\xi|<2^{-L}$. Vice versa, if $f(\xi)=0$, then $|f(x)|<2^{-L}$ for all $x\in\C$ with $|x-\xi|<2^{-L}$.
\end{theorem}

\begin{proof}
For the proof, we mainly combine known results~\cite{MSW-rootfinding2013}, however, we aim to stress the fact that the following computations are necessary to derive an $L$ of size $O(n(\log n+\tau))$. Namely, the literature only provides comparable bounds for square-free polynomials, whereas, for arbitrary polynomials, the existing bounds are of size $\tilde{O}(n^2+n\mu)$. This is mainly due to the fact that the known bounds for square-free polynomials are directly applied to the square-free part, and, in general, the square-free part of an integer polynomial of magnitude $(n,\mu)$ is of magnitude $(n,O(n+\mu))$.

Let $F(x)=f(x)\cdot g(x)=F_N\cdot\prod_{j=1}^N (x-z_j)$, where $z_1,\ldots,z_N$ denote the complex roots of $F$. Then, $F$ has degree $N\le 2n$ and its coefficients are integers of absolute value $2^{\tau_F}$ with $\tau_F<2(\mu+\log n)$. Now, suppose that $F$ has exactly $r_0$, with $1\le r_0\le \deg F$, distinct complex roots $\xi_1$ to $\xi_{r_0}$ with multiplicities $m_1$ to $m_{r_0}$, respectively. From the proof of~\cite[Theorem 5]{MSW-rootfinding2013}, we conclude that
\[
\prod_{i=1}^{r_0}\min\left(1,\frac{|F^{(m_i)}(\xi_i)|}{|F_N|\cdot m_i!}\right)\ge \left(2^{3\tau_F+2\cdot\log N+1}\cdot \operatorname{Mea}(F)\right)^{-N},
\]
where $\operatorname{Mea}(F)=|F_N|\cdot\prod_{i=1}^{r_0}\max(1,|\xi_i|)^{m_i}$ denotes the Mahler Measure of $F$ and $F^{(m)}(x):=\frac{d^m F(x)}{dx^m}$ the $m$-th derivative of $F$. Since $\operatorname{Mea}(F)\le \|F\|_2\le\sqrt{N+1}\cdot 2^{\tau_F}$, a simple computation shows that
\begin{align}
\prod_{i=1}^{r_0}\min\left(1,\frac{|F^{(m_i)}(\xi_i)|}{|F_N|\cdot m_i!}\right)> 2^{-24n(\mu+\log n)}.\label{bound1}
\end{align}
Now, assume that $\xi=\xi_i$ is a root of $g$ and that $f(\xi)\neq 0$. Then, it follows that
\[
|f(\xi)|=\frac{|F^{(m_i)}(\xi_i)|}{|g^{(m_i)}(\xi_i)|}> \frac{2^{-24n(\mu+\log n)}}{(n+1)\cdot 2^\tau \cdot |\xi_i|^n}>2^{-28n(\mu+\log n)},
\]
where we used that $\xi_i$ is a root of $g$ of multiplicity $m_i$ and $|\xi_i|<2^{\mu+1}$ for all $i$. Hence, if $w:=|x-\xi|<2^{-L}$, then
\begin{align*}
|f(x)|&=\left|f(\xi)+\frac{f'(\xi)}{1!}\cdot w+\cdots+\frac{f^{(n)}(\xi)}{n!}\cdot w^n\right|\\
&\ge |f(\xi)|-w\cdot n^2\cdot 2^{\mu}\cdot 2^{n(\mu+1)}\ge 2^{-32(n(\mu+\log n))}.
\end{align*}
Vice versa, if we assume that $f(\xi)=0$, then $|f(x)|<w\cdot n^2\cdot 2^{n(\mu+1)}<2^{-64n(\mu+\log n)}$ for all $x$ with $|x-\xi|\le w\le 2^{-L}$. This proves the second claim. For the first claim, let $\xi_i$ and $\xi_j$ be any two distinct roots of $F$. Then, we conclude from (\ref{bound1}) that
\begin{align}\label{boundonsep2}
2^{-24n(\mu+\log n)}&<\frac{|F^{(m_i)}(\xi_i)|}{|F_N|\cdot m_i!}=\prod_{l\neq i}|\xi_i-\xi_l|^{m_l}\\
&= |\xi_i-\xi_j|^{m_j}\cdot \prod_{l\neq i,j}|\xi_i-\xi_l|^{m_l}\le |\xi_i-\xi_j|^{m_j}\cdot 2^{2N(\tau_F+1)},\nonumber
\end{align}
and thus, the first claim follows.
\end{proof}

\begin{lemma}\label{rationalbound}
Let $a\in\Z$ be a positive integer such that $\sqrt{a}\notin\Z$, and let $p$ and $q$ be arbitrary integers. Then, $|\sqrt{a}-\frac{p}{q}|>\frac{1}{32a^2\cdot q^2}$.
\end{lemma}
\begin{proof}
We consider the continued fraction expansion of $\sqrt{a}$:
\[
\sqrt{a}=a_0+\frac{1}{a_1+\frac{1}{a_2+\frac{1}{\cdots}}},
\]
with non-negative integers $a_0,a_1,\ldots$. For short, we write $\sqrt{a}=[a_0;a_1,a_2,\ldots]$. Abbreviating the continued fraction expansion at $a_n$ yields the corresponding approximation $\frac{p_n}{q_n}=[a_0;a_1,\ldots,a_n]$, where $p_n$ and $q_n$ are co-prime integers. The recursive relation between the denominators $q_n$ and the values $a_n$ is given as follows:
\[
q_n=a_n\cdot q_{n-1}+q_{n-2}.
\]
Furthermore, it holds that the sequence $(a_1,a_2,\ldots)$ is periodic\footnote{More precisely, there exists a $k\in \N$, with $k\le 2a$ such that $(a_1,a_2,\ldots,a_{k})=(a_1,a_2,a_3\ldots a_3,a_2,a_1,2a_0)$ and $a_{k+i}=a_i$ for all $i\in\N_{\ge 1}$.}, and each value $a_i$ is smaller than $2\sqrt{a}$; see~\cite{cohn1977length}. Since the denominators $q_n$ are monotonously increasing, there exists an $n_0$ with $q_{n_0-1}<q\le q_{n_0}$. Then, from the above recursion, we conclude that $q_{n_0}<2\sqrt{a}\cdot q_{n_0-1}+q_{n_0-2}<4\sqrt{a}\cdot q$ and $q_{n_0+1}=a_{n_0+1}\cdot q_{n_0}+q_{n_0-1}<16\cdot a\cdot q$.
According to~\cite[I, \S 2, Theorem~5]{lang1995introduction}, we have
$$\left|\sqrt{a}-\frac{p_n}{q_n}\right|>\frac{1}{2\cdot q_n\cdot q_{n+1}}\quad\text{ for all }n,$$
and thus, 
$$|\sqrt{a}-\frac{p_{n_0}}{q_{n_0}}|>\frac{1}{32a\sqrt{a}\cdot q^2}>\frac{1}{32a^2\cdot q}.$$
Our claim now follows from the fact that, for all $n\ge 1$, the continued fraction approximation $p_n/q_n$ minimizes the distance between $\sqrt{a}$ and any rational value $p/q$ with $q\le q_n$.
\end{proof}   
\newpage

\end{document}